\documentclass[a4paper,12pt]{article}%
\usepackage{amssymb}
\usepackage{amsfonts}
\usepackage{amsmath}
\usepackage{color}
\usepackage{graphicx}%
\providecommand{\U}[1]{\protect\rule{.1in}{.1in}}
\newtheorem{theorem}{Theorem}

\newtheorem{definition}[theorem]{Definition}

\newtheorem{lemma}[theorem]{Lemma}

\newtheorem{remark}[theorem]{Remark}

\newenvironment{proof}[1][Proof]{\noindent\textbf{#1.} }{\ \rule{0.5em}{0.5em}}
\numberwithin{equation}{section}
\newcommand{\e}{\varepsilon}

\begin{document}

\title{Varifold solutions of a sharp interface limit \\of a diffuse interface model for tumor growth}
\author{Elisabetta Rocca\\\textsl{Dipartimento di Matematica, Universit\`a degli Studi di Pavia}\\{Via Ferrata 5, 27100 Pavia, Italy}\\\textrm{E-mail:~~\texttt{elisabetta.rocca@unipv.it}}
\and Stefano Melchionna\\\textsl{University of Vienna, Faculty of Mathematics}\\{Oskar-Morgenstern-Platz 1, 1090 Vienna, Austria}\\\textrm{E-mail:~~\texttt{stefano.melchionna@univie.ac.at}} }
\maketitle

\begin{abstract}
We discuss the sharp interface limit of a diffuse interface model for a
coupled Cahn-Hilliard--Darcy system that models tumor growth when a certain
parameter $\varepsilon>0$, related to the interface thickness, tends to zero.
In particular, we prove that weak solutions to the related initial boundary
value problem tend to varifold solutions of a corresponding sharp interface
model when $\varepsilon$ goes to zero.

\end{abstract}



\vspace{2mm}

\noindent\textbf{AMS (MOS) subject clas\-si\-fi\-ca\-tion:} 35R35, 35Q92,
35K46, 49J40, 92B05.

\section{Introduction}

The present contribution is devoted to the study of the relations between a
diffuse and a sharp interface Cahn-Hilliard-Darcy model for tumor growth.

The morphological evolution of a growing solid tumor is the result of the
dynamics of a complex system that includes many nonlinearly interacting
factors, such as cell-cell and cell-matrix adhesion, mechanical stress, cell
motility and angiogenesis just to name a few. It is clear that mathematics
could make a huge contribution to many areas of experimental cancer
investigation since there is now a wealth of experimental data which requires
systematic analysis. At the current stage of cancer research, most of the
mathematical models are built and developed from the following three
perspectives: discrete (microscopic), continuous (macroscopic), and hybrid
(micro-macroscopic). Numerous mathematical models have been developed to study
various aspects of tumor progression and this has been an area of intense
research (see the recent reviews by Fasano et al. \cite{FBG06}, Graziano and
Preziosi \cite{GP07}, Friedman et al. \cite{FBM07}, Bellomo et al.
\cite{BLM08}, Cristini et al. \cite{CEtAl08}, and Lowengrub et al.
\cite{LEtAl10}). The existing models can be classified into two main
categories: continuum models and discrete models. We concentrate on the former
ones. This category can be subsequently split in two basic types of models
namely the (classical) \textsl{sharp} interface models, where the interface
between the fluids is modeled as a (sufficiently smooth) surface, and
so-called diffuse interface models, where the \textsl{sharp} interface is
replaced by an interfacial region, where a suitable order parameter ($\phi$ in
what follows) varies smoothly, but with a large gradient between two
distinguished values.

The necessity of dealing with multiple interacting constituents has led, in
particular, to the consideration of diffuse-interface models based on
continuum mixture theory (see, for instance, \cite{CLW09} and references
therein). In the diffuse approach, sharp interfaces are replaced by narrow
transition layers that arise due to differential adhesive forces among the
cell-species. The main advantages of the diffuse interface formulation are:

\begin{itemize}
\item[-] it eliminates the need to enforce complicated boundary conditions
across the tumor/host tissue and other species/species interfaces that would
have to be satisfied if the interfaces were assumed sharp, and

\item[-] it eliminates the need to explicitly track the position of
interfaces, as is required in the sharp interface framework.
\end{itemize}

Then, the natural question arises how diffuse and sharp interface models are
related if a suitable parameter $\varepsilon>0$, which measures the width of
the diffuse interface, tends to zero. There are already some results on this
question, which are based on formally matched-asymptotics calculations
(cf.~the recent work by Garcke et. al. \cite{GLSS}), but so far there are very
few mathematically rigorous convergence results (cf. \cite{RS16}). This is
indeed the aim of the present contribution.

The mathematical technique we exploit here consists mainly in considering the
know results for Cahn-Hilliard equations by \cite{Chen} and try to extend them
to the coupled Cahn-Hilliard-Darcy system (first neglecting the nutrient) in
the spirit of what Abels et al. (cf.~\cite{AL} and also \cite{AR}) did for a
two-phase fluid model. The problem of dealing with a complete tumor-growth
model coupling Cahn-Hilliard equation for the tumor phase with a non-zero
source, Darcy law for the velocity, and a reaction-diffusion equation for the
nutrient (cf., e.g., \cite{DFRSS} or \cite{GLSS}) is still open.

Other techniques could also be investigated. For example, recently in
\cite{RS16} the authors exploited Gamma convergence tools for Gradient Flows
systems in order to prove the passage from diffuse to sharp interfaces for a
variant of a different tumor growth model proposed in \cite{HZO} (cf.~also
\cite{Hil}) where the velocity field is not considered and a coupled
Cahn-Hilliard-Reaction-Diffusion system is analyzed. It is worth mentioning
that a Gamma-convergence approach cannot be applied to the problem considered
in this paper due to the lack of gradient structure of system under consideration.

The initial boundary value problem we are interested in here is indeed the
following one:%
\begin{align}
\partial_{t}\phi-\Delta\mu+\nabla\cdot(u\phi) &  =0\text{ in }\Omega
\times(0,\infty)\text{,}\label{diff mod ch}\\
\mu &  =-\varepsilon\Delta\phi+\frac{1}{\varepsilon}F^{\prime}(\phi)\text{ in
}\Omega\times(0,\infty)\text{,}\label{diff mod chem}\\
u &  =-\nabla P+\mu\nabla\phi\text{ in }\Omega\times(0,\infty)\text{,}%
\label{diff mod vel}\\
\nabla\cdot u &  =0\text{ in }\Omega\times(0,\infty)\text{,}%
\label{diff mod incomp}\\
\nu\cdot\nabla\phi &  =\nu\cdot\nabla\mu=\nu\cdot u=0\text{ on }\partial
\Omega\times(0,\infty)\text{,}\label{diff mod bc}\\
\phi(0) &  =\phi_{0,\varepsilon}\text{ in }\Omega\text{,}\label{diff mod ic}%
\end{align}
where $\Omega$ is a bounded subset of $%
\mathbb{R}
^{d}$ with a smooth boundary $\partial\Omega$, $\nu$ denotes the outward unit
normal vector to $\partial\Omega$, $F$ is a double-well potential with minima
in $-1$ and $1$, e.g. $F(r)=\frac{1}{8}(1-r^{2})^{2}$, and $\varepsilon$ is a
small positive parameter related to the interface thickness. Moreover,
$\phi_{0,\varepsilon}$ is a family of approximating initial data which satisfy
a well-preparedness condition (see below). The dynamics of the {phase
variable} $\phi$ (and of the chemical potential $\mu$) is regulated by the
convective Cahn-Hilliard equation (\ref{diff mod ch})-(\ref{diff mod chem}).
The {velocity field} $u$ fulfills the Darcy's law (\ref{diff mod vel}) (here
$P$ denotes a pressure) including the so-called Korteweg term $\mu\nabla\phi$.

The PDE system we consider here (as well as some generalizations of it) has
been already studied from the point of view of existence of solutions,
regularity, and long-time behavior in \cite{LTZ13} (cf. also \cite{JWZ} and
\cite{DFRSS} for more general models), while the formal expansion method for
the sharp interface limit has been recently performed in \cite{GLSS} again for
a more complicated system, where also the nutrient variable and chemotaxis
effects have been taken into account.

The matched asymptotic expansion performed in \cite{GLSS} shows, formally,
that system (\ref{diff mod ch})-(\ref{diff mod ic}) converges, for
$\varepsilon\rightarrow0$, to the sharp-interface limit problem given by \
\begin{align}
\phi &  =1\text{ in }\Omega^{T}\text{, }\label{sharp int 1}\\
\phi &  =-1\text{ in }\Omega^{H}\text{,}\label{sharp int -1}\\
2(-V+u\cdot n)  &  =[\nabla\mu]_{H}^{T}\cdot n\text{ on }\Sigma\text{,}%
\label{sharp int V}\\
\mu &  =\sigma k\text{ on }\Sigma\text{,}\label{sharp int mu}\\
\lbrack\mu]_{H}^{T}  &  =0\text{ on }\Sigma\text{,}\label{sharp int mu jump}\\
-\Delta\mu &  =0\text{ in }\Omega^{T}\cup\Omega^{H}\text{,}%
\label{sharp int delta mu}\\
u  &  =-\nabla P\text{ in }\Omega^{T}\cup\Omega^{H}\text{,}
\label{sharp int u}\\
\nabla\cdot u  &  =0\text{ in }\Omega^{T}\cup\Omega^{H}\text{,}%
\label{sharp nt div u}\\
\lbrack u]_{H}^{T}\cdot n  &  =0\text{ on }\Sigma\text{,}%
\label{sharp int u jump}\\
\lbrack P]_{H}^{T}  &  =2\sigma k\text{ on }\Sigma\text{.} \label{sharp int 2}%
\end{align}
Here \textit{tumor region} $\Omega^{T}$ and the \textit{healthy region}
$\Omega^{H}$ are two open and disjoint subset of $\Omega$ separated by a
smooth interface $\Sigma$ which moves with \textit{normal velocity} $V$.
Moreover, $\sigma$ is a constant related to the potential given by
$\sigma=\int_{-1}^{1}\sqrt{\frac{F(r)}{2}}\mathrm{d}r$, $k$ is the mean
curvature of $\Sigma$, $n$ is the outward unit normal to $\Sigma$ pointing
towards $\Omega^{T}$, and $[f]_{H}^{T}$ denotes the jump of $f$ from
$\Omega^{T}$ to $\Omega^{H}$ across the interface $\Sigma$. As for the diffuse
interface case we close the system with boundary and initial conditions%
\begin{align*}
\nu\cdot u  &  =0\text{ on }\partial\Omega\times(0,\infty)\text{,}\\
\nu\cdot\nabla\mu &  =0\text{ on }\partial\Omega\times(0,\infty)\text{,}\\
\Omega^{T}(0)  &  =\Omega_{0}^{T}\text{,}%
\end{align*}
where $\Omega_{0}^{T}$ is the tumor region at the initial time $t=0$.

Our goal is to prove the convergence rigorously. More precisely, in the rest
of the paper we address the following question: under which assumptions on the
potential $F$ do weak solutions of (\ref{diff mod ch})--(\ref{diff mod ic})
converge to weak/generalized solutions of (\ref{sharp int 1}%
)--(\ref{sharp int 2})? We show that if $F$ satisfies proper growth conditions
at infinity, which are fulfilled in particular by the so-called standard
double-well potential $F(r)=\frac{1}{8}(1-r^{2})^{2}$, then the weak
solutions of (\ref{diff mod ch})--(\ref{diff mod ic}) converge to the
so-called varifold solutions of (\ref{sharp int 1})--(\ref{sharp int 2}),
which are defined in the spirit of \cite{Chen} in Section~\ref{main}.

The paper is organized as follows: in Section~\ref{notation} we introduce some
notation and preliminaries we need in the rest of the paper. In
Section~\ref{main} we state our assumptions on the data and the main result of
the paper together with the notion of solutions. Finally, in the last two
Sections~\ref{priori}, \ref{convergence} we prove the main
Theorem~\ref{main thm} by establishing suitable a-priori estimates
(independent of $\varepsilon$) on the solution to (\ref{diff mod ch})--(\ref{diff mod ic})
leading to the passage to the limit as $\varepsilon\rightarrow0$.

\section{Preliminaries and notation}

\label{notation}

In this section we fix the notation and recall some known facts about
functions of bounded variation and varifolds.

Given $\Omega\subset%
\mathbb{R}
^{d}$ a bounded set with a smooth boundary, $d,N\in%
\mathbb{N}
$, $X$ a Banach space with separable dual space $X^{\ast}$, we use the
following notations for these functional spaces.

\begin{itemize}
\item $L^{p}(\Omega)$ and $L^{p}(\Omega,X)$, for $p\in\lbrack1,\infty]$,
denote the standard Lebesgue spaces for scalar and $X$ valued functions, respectively.

\item $C_{0}(\Omega,%
\mathbb{R}
^{N})$ is the closure of compactly-supported continuous functions
$f:\Omega\rightarrow%
\mathbb{R}
^{N}$, in the supremum norm.

\item $C_{0}^{k}(\Omega)$, $k\in%
\mathbb{N}
\cup\{\infty\}$ is the set of $k$-times-differentiable compactly-supported functions.

\item $C^{k}(\bar{\Omega})$, $k\in%
\mathbb{N}
\cup\{\infty\}$ is the set of $k$-times-differentiable functions such that all
derivatives have a continuous extension on $\bar{\Omega}$.

\item $C_{0,\operatorname{div}}^{\infty}(\Omega)=\{f\in C_{0}^{\infty}%
(\Omega):\nabla\cdot f=0\}$ and $L_{\operatorname{div}}^{2}(\Omega
)=\overline{C_{0,\operatorname{div}}^{\infty}(\Omega)}^{L^{2}(\Omega)}$.

\item $L_{\mathrm{loc}}^{p}(0,\infty;X)$ for $p\in\lbrack1,\infty)$ denotes
the space of all measurable functions $f:(0,\infty)\rightarrow X$ such that
$f\in L^{p}(0,t;X)$ for all $t>0$.

\item $M(\Omega;%
\mathbb{R}
^{N})$ for $N\in%
\mathbb{N}
$, denotes the space of all finite $%
\mathbb{R}
^{N}$-valued Radon measures. $M(\Omega;%
\mathbb{R}
)=:M(\Omega)$.

\item $\mathrm{BV}(\Omega)$ is the space of functions of bounded variations.

\item $L_{\omega\ast}^{\infty}(\Omega;X^{\ast})$ denotes the space of all
functions $f:\Omega\rightarrow X^{\ast}$ that are weakly* measurable and
essentially bounded.
\end{itemize}

Given $f\in\mathrm{BV}(\Omega)$ we denote by $Df$ its distributional gradient
and by $|Df|$ the Radon measure generated by%
\[
|Df|(A)=\sup_{Y\in C_{0}(A;%
\mathbb{R}
^{d}):|Y|\leq1}\int_{A}f\nabla\cdot Y\mathrm{d}x\text{, \ \ \ for all }A\text{
open in }\Omega\text{.}%
\]
Moreover, one can show (cf., e.g., \cite{Federer}) that there exists a
$|Df|$-measurable unit vector valued function $n$ such that $Df=n|Df|$,
$|Df|$-a.e.. We recall that%
\[
\mathrm{BV}(\Omega)=\{f\in L^{1}(\Omega):Df\in M\left(  \Omega;%
\mathbb{R}
^{d}\right)  \}
\]
and
\[
\left\Vert f\right\Vert _{\mathrm{BV}(\Omega)}=\left\Vert f\right\Vert
_{L^{1}(\Omega)}+\left\Vert Df\right\Vert _{M\left(  \Omega;%
\mathbb{R}
^{d}\right)  }=\left\Vert f\right\Vert _{L^{1}(\Omega)}+|Df|(\bar{\Omega
})\text{.}%
\]
Let $E$ be a set in $\Omega$. If the characteristic function $\chi_{E}$
belongs to $\mathrm{BV}(\Omega)$, then we say that $E$ has \textit{finite
perimeter} and we denote $D\chi_{E}=n_{E}|D\chi_{E}|$. Note that, if $\partial
E$ is smooth, then $n_{E}$ is the unit inward norm to $\partial E$. Moreover,
we recall that there exists a separable Banach space $X$ such that its dual
space coincide with $\mathrm{BV}(\Omega)$, (cf. \cite{AFP}). As a consequence
the space $L_{\omega\ast}^{\infty}\left(  0,s;\mathrm{BV}(\Omega)\right)
=\left(  L^{1}(0,s;X)\right)  ^{\ast}$ is well defined.

Let now
\[
P=S^{d-1}/\{\nu,-\nu\}
\]
be the set of unit normals of unoriented $(d-1)$-dimensional hyperplanes in $%
\mathbb{R}
^{d}$. A \textit{varifold} $V$ is a Radon measure on $\Omega\times P$. We
define the mass measure $\left\Vert V\right\Vert $ as the Radon measure on
$\Omega$ given by%
\[
\left\Vert V\right\Vert \left(  A\right)  =\int\int_{A\times P}\mathrm{d}%
V(x,p)\text{ \ \ for all }A\text{\ open in }\Omega\text{.}%
\]
The\emph{ }\textit{first variation} $\delta V$ of a varifold $V$ is the linear
functional on $C_{0}^{1}(\Omega;\mathbb{R}^{d})$ defined by
\[
\left\langle \delta V,Y\right\rangle :=\int\int_{\Omega\times P}\nabla
Y:\left(  I-p\otimes p\right)  \mathrm{d}V(x,p)\text{ \ \ \ for all }Y\in
C_{0}^{1}(\Omega;%
\mathbb{R}
^{d})
\]
and its \textit{mean curvature vector} $H$ (wherever it exists) is a
$\left\Vert V\right\Vert $-measurable vector-valued function on $\Omega$
defined by%
\[
-\left\langle \delta V,Y\right\rangle =\left\langle \left\Vert V\right\Vert
,H\cdot Y\right\rangle =\int_{\Omega}\left(  Y(x)\cdot H(x)\right)
\mathrm{d}\left\Vert V\right\Vert \left(  x\right)  \text{ \ \ \ for all }Y\in
C_{0}^{1}(\Omega;%
\mathbb{R}
^{d})\text{.}%
\]

\section{Assumptions and main results}

\label{main}

In this section we introduce the main assumptions on the problem data and the
statement of the main results.

Let the potential $F$ be such that $F\in C^{3}(%
\mathbb{R}
)$, $F(\pm1)=0$, and $F(r)>0$ if $r\neq\pm1$. Moreover, let exist
constants $c_{0},C_{c_{0}}>0$, $p\geq4$ such that $F^{\prime\prime}(r)\geq
C_{c_{0}}|r|^{p-2}$ for all $r$ such that $|r|\geq1-c_{0}$. An
example of potential $F$ satisfying the above assumption is the classical
double-well potential $F(r)=\frac{1}{8}(1-r^{2})^{2}$.

\begin{remark}
Note that the same conditions with $p\geq3$ are assumed in \emph{\cite{AL}},
where the authors consider the sharp interface limit of a Cahn-Hilliard
equation coupled with a Navier-Stokes equation, instead of the Darcy's law
\emph{(\ref{diff mod vel})}, for the velocity field. Here we need stronger
coercivity assumptions on $F$ as solutions $u$ to the Darcy's law
\emph{(\ref{diff mod vel})} are, in general, less regular than solutions to
the Navier-Stokes equation.
\end{remark}

We also assume uniform boundedness of the initial energy. More precisely, let
$\phi_{0,\varepsilon}\in H^{1}(\Omega)\cap L^{p}(\Omega)$ be such that there
exists a positive constant $E_{0}$ satisfying%
\begin{equation}
E_{\varepsilon}(\phi_{0,\varepsilon})\leq E_{0}\text{,}%
\label{well preparedness}%
\end{equation}
where the \textit{energy} functional $E_{\varepsilon}$ is defined by%
\begin{equation}
E_{\varepsilon}(\phi)=\int_{\Omega}(\varepsilon \frac{1}{2} |\nabla\phi|^{2}+\frac
{1}{\varepsilon}F(\phi))\mathrm{d}x\text{.}\label{energy}%
\end{equation}
Finally, we ask the initial tumor mass to be independent of $\varepsilon$,
namely
\[
\bar{\phi}_{0,\varepsilon}=\frac{1}{|\Omega|}\int_{\Omega}\phi_{0,\varepsilon
}\mathrm{d}x=m_{0}\in(-1,1)\text{.}%
\]

Before stating our main result, let us rigorously define solutions to system
(\ref{diff mod ch})-(\ref{diff mod ic}) and system (\ref{sharp int 1}%
)-(\ref{sharp int 2}).

\begin{definition}
[Weak solutions to (\ref{diff mod ch})-(\ref{diff mod ic})] \label{def sol diff} We call
$(\phi_{\varepsilon},\mu_{\varepsilon},u_{\varepsilon})$ a \emph{weak
solution} to system \emph{(\ref{diff mod ch})-(\ref{diff mod ic})} if these
functions belong to the regularity class:%
\begin{align*}
\phi_{\varepsilon} &  \in C^{0}([0,\infty);H^{1}(\Omega))\cap L_{\mathrm{loc}%
}^{2}(0,\infty;H^{2}(\Omega))\cap H^1_{\mathrm{loc}}(0,\infty;L^2(\Omega) )\text{,}\\
\mu_{\varepsilon} &  \in L_{\mathrm{loc}}^{2}(0,\infty;L^{2}(\Omega))\text{,
\ \ }\nabla\mu_{\varepsilon}\in L^{2}(0,\infty;L^{2}(\Omega))\text{,}\\
u_{\varepsilon} &  \in L^{2}(0,\infty;L_{\operatorname{div}}^{2}(\Omega)),
\end{align*}
and the following integral identities hold:%
\begin{align}
& \int_{0}^{t}\int_{\Omega}\left(  \phi_{\varepsilon}\partial_{t}\psi
+\phi_{\varepsilon}u_{\varepsilon}\cdot \nabla\psi-\nabla\mu_{\varepsilon
}\cdot\nabla\psi\right)  \mathrm{d}x\mathrm{d}s\nonumber\\
& =\int_{\Omega}\phi_{\varepsilon}(t)\psi(t)\mathrm{d}x-\int_{\Omega}%
\phi_{0,\varepsilon}\psi(0)\mathrm{d}x\text{,} \label{weak form diff prob}
\end{align}
for all $\psi\in C_{0}^{\infty}([0,t]\times\Omega)$, $t>0$, and 
\begin{equation}
\frac{\mathrm{d}}{\mathrm{d}t}E_{\varepsilon}(\phi_{\varepsilon})+\int
_{\Omega}|\nabla\mu_{\varepsilon}|^{2}\mathrm{d}x+\int_{\Omega}|u_{\varepsilon
}|^{2}\mathrm{d}x=0, \label{energy2}
\end{equation}
where
\begin{align}
\mu_{\varepsilon} &  =-\Delta\phi_{\varepsilon}+\frac{1}{\varepsilon}%
F^{\prime}(\phi_{\varepsilon})\text{ a.e. in }\Omega\times\lbrack
0,\infty)\text{,}\label{def mu eps}\\
u_{\varepsilon} &  =-\nabla P+\mu_{\varepsilon}\nabla\phi_{\varepsilon}\text{
a.e. in }\Omega\times\lbrack0,\infty)\text{,}\\
\nu\cdot\nabla\phi_{\varepsilon} &  =0\text{ a.e. on }\partial\Omega
\times\lbrack0,\infty) \text{,} \\
E_\varepsilon(\phi_\varepsilon) &= \int_\Omega e_\varepsilon (\phi_\varepsilon) \mathrm{d}x = \int_\Omega \varepsilon
\frac{1}{2}|\nabla\phi_{\varepsilon}
|^{2}+\frac{1}{\varepsilon}F(\phi_{\varepsilon}) \mathrm{d}x \text{ a.e. in } [0,\infty).
\end{align}
\end{definition}

\begin{definition}
[Varifold solutions to (\ref{sharp int 1})-(\ref{sharp int 2})]%
\label{def sharp int sol}Let $\Omega_{0}^{T}$ be a set of finite perimeter.
Then, $(u,\Omega^{T},\mu,V)$ is called a \emph{varifold solution} to
\emph{(\ref{sharp int 1})-(\ref{sharp int 2})} if the following conditions are satisfied:

\begin{enumerate}
\item $u\in L^{2}(0,\infty;L^{2}(\Omega))$, $\mu\in L_{\mathrm{loc}}%
^{2}(0,\infty;L^{2}(\Omega))$, $\nabla\mu\in L^{2}(0,\infty;L^{2}(\Omega))$.

\item $\Omega^{T}$ can be decomposed as $\Omega^{T}=\cup_{t\geq0}\Omega
_{t}^{T}\times\{t\}$, where $\Omega_{t}^{T}$ is a measurable subset of
$\Omega$. Furthermore,
\[
\chi_{\Omega^{T}}\in C\left(  [0,\infty);L^{1}(\Omega)\right)  \cap
L_{w^{\ast}}^{\infty}(0,\infty;\mathrm{BV}(\Omega))
\]
and $|\Omega_{t}^{T}|=|\Omega_{0}^{T}|$ for all $t\geq0$.

\item $V$ is a Radon measure on $\bar{\Omega}\times P\times(0,\infty)$ such
that $V=V^{t}\mathrm{d}t$ where $V^{t}$ is a Radon measure on $\bar{\Omega
}\times P$ for almost all $t\in(0,\infty)$. Moreover, for a.a. $t\in
(0,\infty)$, $V^{t}$ admits the representation
\begin{equation}
\int_{\bar{\Omega}\times P}\psi(x,p)\mathrm{d}V^{t}(x,p)=\sum_{i=1}^{d}%
\int_{\bar{\Omega}}b_{i}^{t}(x)\psi(x,p_{i}^{t}(x))\mathrm{d}\lambda
^{t}(x)\label{def V}%
\end{equation}
for all $\psi\in C\left(  \bar{\Omega}\times P\right)  $, some Radon measure
$\lambda^{t}$ on $\bar{\Omega}$, and some $\lambda^{t}$-measurable functions
$b_{i}^{t}$, $p_{i}^{t}$ with values in $%
\mathbb{R}
$ and $P$ respectively such that
\[
0\leq b_{i}^{t}\leq1\text{, }\sum_{i=1}^{d}b_{i}^{t}\geq1\text{, }\sum
_{i=1}^{d}p_{i}^{t}\otimes p_{i}^{t}=I\text{ \ \  }\lambda^{t}\text{-a.e.,}%
\]
and
\begin{equation}
\frac{|D\chi_{\Omega_{t}^{T}}|}{\lambda^{t}}\leq\frac{1}{2\sigma}%
\text{.}\label{bound}%
\end{equation}

\item For every $t>0$ and every $\psi\in C^{\infty}_0(\lbrack
0,t] \times{\Omega})$,
\begin{align}
& \int_{0}^{t}\int_{\Omega}[2\chi_{\Omega_{s}^{T}}\partial_{t}\psi-\nabla
\mu\nabla\psi+2\chi_{\Omega_{s}^{T}}u \cdot \nabla\psi]\mathrm{d}%
x\mathrm{d}s\nonumber\\
& =\int_{\Omega}2\chi_{\Omega_{t}^{T}}\psi(t)\mathrm{d}x-\int_{\Omega}%
2\chi_{\Omega_{0}^{T}}\psi(0)\mathrm{d}x.\label{weak form sharp int}%
\end{align}

\item For every $t>0$ and every $Y\in C_{0}^{1}(\Omega,%
\mathbb{R}
^{d})$,%

\begin{equation}
-\left\langle \mathrm{D}\chi_{\Omega_{t}^{T}},\mu Y\right\rangle =\int
_{\Omega}\chi_{\Omega_{t}^{T}}\nabla\cdot(\mu Y)\mathrm{d}x=\frac{1}%
{2}\left\langle \delta V^{t},Y\right\rangle \text{.}\label{first variation V}%
\end{equation}

\item For every $0\leq\tau<t$,
\begin{equation}
\lambda^{t}(\bar{\Omega})+\int_{\tau}^{t}\int_{\Omega}|\nabla\mu
|^{2}\mathrm{d}x\mathrm{d}s+\int_{\tau}^{t}\int_{\Omega}|u|^{2}\mathrm{d}%
x\mathrm{d}s\leq\lambda^{\tau}(\bar{\Omega})\text{.}
\label{nergy ineq sharp int}%
\end{equation}

\item For every $t>0$ and every $\varphi\in C_{0,\operatorname{div}}^{\infty
}(\Omega)$, we have
\begin{equation}
\int_{0}^{t}\int_{\Omega}u\varphi\mathrm{d}x\mathrm{d}s=\int_{0}^{t}%
\int_{\Sigma_{s}}2\mu\varphi\mathrm{d}S\mathrm{d}s\text{,}\label{weak form u}%
\end{equation}
where $\Sigma_{t}=\partial\Omega_{t}^{T}\setminus\partial\Omega$.
\end{enumerate}
\end{definition}

Let us postpone some remarks and comments on the definition of solutions and
state our main result.

\begin{theorem}
[Sharp interface limit]\label{main thm}Let the above assumptions be satisfied.
Then, there exists a sequence $\varepsilon\rightarrow0$ such that the
following holds.

\begin{enumerate}
\item There exists $\Omega^{T}=\cup_{t\geq0}\Omega_{t}^{T}\times
\{t\}\subset\Omega\times\lbrack0,\infty)$ such that
\begin{equation}
\phi_{\varepsilon}\rightarrow-1+2\chi_{\Omega^{T}}\text{ a.e. in~}\Omega
\times\lbrack0,\infty)\text{ and in }C^{\frac{1}{17}}([0,t);L^{2}%
(\Omega))\text{ for any } t\geq0.\label{conv phi}%
\end{equation}

\item There exists $\mu\in L_{\mathrm{loc}}^{2}(0,\infty,;L^2(\Omega))$ such
that $\nabla \mu\in L^{2}(0,\infty,;L^2(\Omega))$ and
\[
\mu_{\varepsilon}\rightarrow\mu\text{ weakly in }L_{\mathrm{loc}}^{2}%
(0,\infty,;H^{1}(\Omega))\text{.}%
\]

\item There exists $u\in L^{2}(0,\infty;L_{\operatorname{div}}^{2}(\Omega))$
such that%
\[
u_{\varepsilon}\rightarrow u\text{ weakly in }L^{2}(0,\infty
;L_{\mathrm{\operatorname{div}}}^{2}(\Omega))\text{.}%
\]

\item There exist a Radon measure $\lambda$ and measures $\lambda_{ij}$,
$i,j\in\{1,...,d\}$, on $\bar{\Omega}\times\lbrack0,\infty)$ such that%
\begin{align}
e_{\varepsilon}(\phi_{\varepsilon})\mathrm{d}x\mathrm{d}t &  \rightarrow
\lambda\text{ as a Radon measure }\bar{\Omega}\times\lbrack0,\infty
)\text{,}\nonumber\\
&  \text{i.e. weakly star in }M(\Omega,%
\mathbb{R}
)\text{,}\label{conv energy density}\\
\varepsilon\partial_{x_{i}}\phi_{\varepsilon}\partial_{x_{j}}\phi
_{\varepsilon}\mathrm{d}x\mathrm{d}t &  \rightarrow\lambda_{ij}\text{ as a
measure on }\bar{\Omega}\times\lbrack0,\infty)\text{,}\nonumber\\
\text{ for }i,j &  \in\{1,...,d\}\text{,}\label{conv grad phi}%
\end{align}
where $e_{\varepsilon}(\phi_{\varepsilon})$ denotes the \emph{energy density}:%
\[
e_{\varepsilon}(\phi_{\varepsilon})=\varepsilon \frac{1}{2}|\nabla\phi_{\varepsilon}%
|^{2}+\frac{1}{\varepsilon}F(\phi_{\varepsilon})\text{.}%
\]

\item There exists a Radon measure $V=V^{t}\mathrm{d}t$ on $\bar{\Omega}\times
P\times\lbrack0,\infty)$ such that $(u, \Omega^T, \mu,V)$ is a Varifold solution of
\emph{(\ref{sharp int 1})-(\ref{sharp int 2})} in the sense of Definition
\emph{\ref{def sharp int sol}}, with $\mathrm{d}\lambda^{t}(x)\mathrm{d}%
t=\mathrm{d}\lambda(x,t)$ (where $\lambda^{t}$ as in
\emph{(\ref{nergy ineq sharp int})} and $\lambda$ as in
\emph{(\ref{conv energy density})}) and with $\sigma=\int_{-1}^{1}\sqrt
{\frac{F(r)}{2}}\mathrm{d}r$. Moreover,
\begin{equation}
\int_{0}^{t}\left\langle \delta V^{s},Y\right\rangle \mathrm{d}s=\int_{0}%
^{t}\int_{\Omega}\nabla Y:\left[  \mathrm{d}\lambda(x,s)I-\left(
\mathrm{d}\lambda_{ij}(x,s)\right)  _{d\times d}\right]
\label{first variation V2}%
\end{equation}
for all $Y$ in $C_{0}^{1}(\Omega\times\lbrack0,t];%
\mathbb{R}
^{d})$ and for all $t>0$.
\end{enumerate}
\end{theorem}

\begin{remark}
Let us now comment on the notion of solutions introduced in Definition
\emph{\ref{def sol diff}} and Definition \emph{\ref{def sharp int sol}}.

\begin{enumerate}

\item The weak formulation \emph{(\ref{weak form diff prob})} is derived
by testing \emph{(\ref{diff mod ch})} with some $\psi \in C^\infty_0 ([0,t]\times \Omega)$,
integrating by parts in time and space, and using the boundary and the initial conditions.
We remark that, as $\phi_\varepsilon \in H^1 (0,t;L^2(\Omega))$,
relation \emph{(\ref{weak form diff prob})} can be equivalently rewritten as
\begin{align*}
&\int_\Omega (-\partial_t \phi_\varepsilon \psi + \phi_\varepsilon u_\varepsilon \cdot \nabla\psi
-\nabla\mu_\varepsilon \cdot\nabla\psi){\rm{d}} x =0 \text{ a.e. in } (0,t),  \\
&\phi_\varepsilon (0) = \phi_{0,\varepsilon}.
\end{align*}

\item The energy identity \emph{(\ref{energy})} can be formally obtained by testing equation \emph{(\ref{diff mod ch})}
by $\mu_\varepsilon$ and \emph{(\ref{diff mod chem})} by $\partial_t \phi _\varepsilon$, comparing the two, integrating 
by parts (taking into account the boundary conditions) and using \emph{(\ref{diff mod vel})}.

\item As stated in Theorem \emph{\ref{main thm}}, $\lambda^{t}(\bar
{\Omega})$ is the limit of the energy functional $E_{\varepsilon}%
(\phi_{\varepsilon}(t))$ as $\varepsilon\rightarrow0$. The energy functional for
the sharp interface problem is instead given by the interfacial energy:
$2\sigma|\chi_{\Omega_{t}^{T}}|(\Omega)$. A natural question is how the two
relate. Modica and Mortola \emph{\cite{Mo-Mo}} and Sternberg \emph{\cite{Ste}}
proved that the functional $E_{\varepsilon}$ converge to $2\sigma|\chi
_{\Omega^{T}}|(\Omega)$ in the Gamma-convergence sense with respect to the
topology of $L^{1}(\Omega)$. As a consequence of this result and of
convergence \emph{(\ref{conv phi})}, we have that
\begin{equation}
\lambda^{t}(\bar{\Omega})=\lim_{\varepsilon\rightarrow0}E_{\varepsilon}%
(\phi_{\varepsilon}(t))=\liminf_{\varepsilon\rightarrow0}E_{\varepsilon}%
(\phi_{\varepsilon}(t))\geq2\sigma|\chi_{\Omega_{t}^{T}}|(\Omega)\text{.}%
\label{ineq energy}%
\end{equation}
A second approach to obtain inequality \emph{(\ref{ineq energy})} is the
following. Consider the relation%
\begin{equation}
\varepsilon\int_{\Omega}|\nabla\phi_{\varepsilon}|^{2}\mathrm{d}x=\int
_{\Omega}e_{\varepsilon}(\phi_{\varepsilon})\mathrm{d}x+\int_{\Omega}%
\xi_{\varepsilon}(\phi_{\varepsilon})\mathrm{d}x\text{,}\label{gg}%
\end{equation}
where the \emph{discrepancy density} $\xi_{\varepsilon}$ is given by
$\xi_{\varepsilon}(\phi_{\varepsilon})=\varepsilon/2\left\vert \nabla
\phi_{\varepsilon}\right\vert ^{2}-1/\varepsilon F(\phi_{\varepsilon})$. We
will prove that the discrepancy measure is nonpositive in the limit
$\varepsilon\rightarrow0$, namely $\int_{\Omega}\left(  \xi_{\varepsilon}%
(\phi_{\varepsilon})\right)  ^{+}\mathrm{d}x\rightarrow0$ as $\varepsilon
\rightarrow0$ (see Lemma \emph{\ref{lemma estimate for discrepancy}}). This
yields, by passing to the limit as $\varepsilon\rightarrow0$ in
\emph{(\ref{gg})}, inequality \emph{(\ref{ineq energy})}. Note that, in
general, it is not possible to prove equality in \emph{(\ref{ineq energy})}
even in the simpler case $u=0$, (cf. Section 2.4 of \emph{\cite{Chen}}). For
example, a strict inequality holds true in case the initial data develop a
so-called \emph{phantom interface}, i.e.,
\[
2\left\vert D\chi_{\Omega_{0}^{T}}\right\vert (\Omega)=\left\vert
D\lim_{\varepsilon\rightarrow0}\phi_{0,\varepsilon}(t)\right\vert
(\Omega)<\liminf_{\varepsilon\rightarrow0}\left\vert D\phi_{0,\varepsilon
}\right\vert (\Omega)\text{.}%
\]
However, in the case $u=0$, under some additional assumptions, e.g., radial
symmetry of the solutions \emph{\cite{Chen}} or limit \emph{equipartition of
the energy: }$\int_{\Omega}\left(  \xi_{\varepsilon}(\phi_{\varepsilon
})\right)  \mathrm{d}x\rightarrow0$ (which holds true if $d\leq3$)
\emph{\cite{Le}}, it is possible to show equality in \emph{(\ref{ineq energy}%
)}. Let us mention that the techniques used in \emph{\cite{Le}} strongly rely on
the gradient-flow structure of equation \emph{(\ref{diff mod ch}%
)-(\ref{diff mod chem})} in the case $u=0$. Thus, it seems hard to generalize
that result to the system under consideration.

\item Using the definition of $V$ \emph{(\ref{def V})}, we have that
\[
\mathrm{d}V^{t}(x,p)=\sum_{i=1}^{d}b_{i}^{t}(x)\delta_{p_{i}^{t}(x)}%
\mathrm{d}\lambda^{t}(x)\text{.}%
\]
Thus, by definition of mass measure of a varifold and as a consequence of the
properties of $b_{i}^{t}$, we get
\[
\mathrm{d}\left\Vert V^{t}\right\Vert (x)=\sum_{i=1}^{d}b_{i}^{t}%
(x)\mathrm{d}\lambda^{t}(x)\geq\mathrm{d}\lambda^{t}(x).
\]
Let $H^{t}$ denote the mean curvature vector of $V^{t}$. Then, by definition,
we have, for all $Y\in C_{0}^{1}(\Omega;%
\mathbb{R}
^{d})$,
\begin{align}
-\left\langle \delta V^{t},Y\right\rangle  &  =\left\langle \left\Vert
V^{t}\right\Vert ,H\cdot Y\right\rangle =\int_{\Omega}H(x)\cdot Y(x)\mathrm{d}%
\left\Vert V^{t}\right\Vert (x)\nonumber\\
&  =\int_{\Omega}2\sigma mH(x)\cdot Y(x)|D\chi_{\Omega_{t}^{T}}|(x)\mathrm{d}%
x\text{,}\label{aa}%
\end{align}
where
\begin{equation}
m:=\frac{\mathrm{d}\left\Vert V^{t}\right\Vert (x)}{2\sigma|D\chi_{\Omega
_{t}^{T}}|(x)\mathrm{d}x}\text{.}\label{def m}%
\end{equation}
Note that the two measures $\mathrm{d}\left\Vert V^{t}\right\Vert (x)$ and
$|D\chi_{\Omega_{t}^{T}}|(x)\mathrm{d}x$ are absolutely continuous one with
respect to the other as a consequence of relation \emph{(\ref{bound})}. Moreover,
$m\geq1$. Furthermore, using formula \emph{(\ref{first variation V})}, we
have
\begin{align}
-\left\langle \delta V^{t},Y\right\rangle  &  =-2\int_{\Omega}\chi_{\Omega
_{t}^{T}}\nabla\cdot(\mu Y)\mathrm{d}x=2\int_{\Omega}\mu n_{\Omega_{t}^{T}%
}\cdot Y(x)|D\chi_{\Omega_{t}^{T}}|(x)\mathrm{d}x\nonumber\\
&  =\int_{\Omega}\frac{\mu}{m\sigma}n_{\Omega_{t}^{T}}\cdot Y(x)\mathrm{d}%
\left\Vert V^{t}\right\Vert (x)\text{,}\label{bb}%
\end{align}
where $n_{\Omega_{t}^{T}}$ is the unit vector associated with $D\chi
_{\Omega_{t}^{T}}$ defined as in Section \emph{\ref{notation}}. Comparing
\emph{(\ref{aa})} and \emph{(\ref{bb})}, we deduce
\[
\frac{\mu}{m}n_{\Omega_{t}^{T}}=\sigma H\text{.}%
\]
Multiplying by $n_{\Omega_{t}^{T}}$, as $|n_{\Omega_{t}^{T}}|=1$, we get
\[
\frac{\mu}{m}=\sigma n_{\Omega_{t}^{T}}\cdot H=\sigma k\text{.}%
\]
Here $k:=n_{\Omega_{t}^{T}}\cdot H$ is the so-called generalized mean
curvature. As $m\geq1$, we have that
\[
\mu=m\sigma k\geq\sigma k\text{ on }\Sigma\text{.}%
\]
Thus, relation \emph{(\ref{sharp int mu})} is satisfied up to a multiplicative
constant $m\geq1$ (if $m=1$, we get equation \emph{(\ref{sharp int mu})}). We
remark that, in general, it is not possible to show $m=1$ even in the simpler
case $u=0$ (cf Section 2.4 of \emph{\cite{Chen}}). This is related to a
possible gap between the limit of the energy and the energy of the limit
problem as already discussed above. Indeed, under some growth assumptions on
$\lambda^{t}$, it is possible to show that $\lambda^{t}=\left\Vert
V^{t}\right\Vert $ (cf. \emph{\cite{Chen}}). In this case, thanks to
\emph{(\ref{def m})}, we have
\begin{equation}
\lambda^{t}=2\sigma m|D\chi_{\Omega_{t}^{T}}|\geq2\sigma|D\chi_{\Omega_{t}%
^{T}}|,\label{ff}%
\end{equation}
which is a quantitative version of inequality \emph{(\ref{ineq energy})}. In
particular, this shows that, in the case $\lambda^{t}=\left\Vert
V^{t}\right\Vert $, equality in \emph{(\ref{ineq energy})} and relation $m=1$
are equivalent.

\item From \emph{(\ref{sharp int mu})} and \emph{(\ref{sharp int 2})},
we easily deduce
\begin{equation}
\lbrack P]_{H}^{T}=2\mu\text{ on }\Sigma.\label{relation}%
\end{equation}
Equation \emph{(\ref{weak form u})} is obtained by multiplying equation
\emph{(\ref{sharp int u})} by some $\varphi\in C_{0,\operatorname{div}%
}^{\infty}(\Omega)$, integrating by parts, and using \emph{(\ref{relation})},
and the boundary conditions:
\begin{align*}
\int_{0}^{t}\int_{\Omega_{s}^{T}\cup\Omega_{s}^{H}}u\cdot\varphi
\mathrm{d}x\mathrm{d}s  & =-\int_{0}^{t}\int_{\Omega_{s}^{T}\cup\Omega_{s}%
^{H}}\nabla P\cdot\varphi\mathrm{d}x\mathrm{d}s\\
& =\int_{0}^{t}\int_{\Sigma_{s}}[P]_{H}^{T}\cdot\varphi\mathrm{d}%
S\mathrm{d}s=\int_{0}^{t}\int_{\Sigma_{s}}2\mu\cdot\varphi\mathrm{d}%
S\mathrm{d}s\text{.}%
\end{align*}

\item Equation \emph{(\ref{first variation V})} together with
\emph{(\ref{first variation V2})} imply
\[
\int_{\Omega}2\chi_{\Omega_{t}^{T}}\nabla\cdot(\mu(t)Y)\mathrm{d}%
x=\int_{\Omega}\nabla Y:\left[  \mathrm{d}\lambda(x,t)I-\left(  \mathrm{d}%
\lambda_{ij}(x,t)\right)  _{d\times d}\right] \text{ for a.a. } t>0 .
\]
This relation can be obtained by passing to the limit $\varepsilon
\rightarrow0$ in formula \emph{(\ref{eqq})}. Therefore, equation
\emph{(\ref{first variation V})} stands as a reformulation of identity
\emph{(\ref{def mu eps})} and of condition $\mu_{\varepsilon}\in\frac{\delta
E_{\varepsilon}}{\delta\phi}$ (here $\frac{\delta E_{\varepsilon}}{\delta\phi
}$ denotes the first variation of $E_{\varepsilon}$) in the limit
$\varepsilon\rightarrow0$.

\item Inequality \emph{(\ref{nergy ineq sharp int})} has the meaning of
energy dissipation inequality. It is obtained starting from
 \emph{(\ref{energy identity}) } and by passing to the liminf as $\e \to 0$.
 We remark that equality does not hold in general. Indeed, we have just weak convergence for
$\nabla\mu_{\varepsilon}$ and $u_{\varepsilon}$.

\item Note that, for all $\psi\in C_0^{\infty}([0,t] \times \Omega)$, we
have $\int_{0}^{t}\int_{\Omega^T \cup \Omega^H}\partial_{t}\psi\mathrm{d}x\mathrm{d}%
s=\int_{\Omega^T \cup \Omega^H}\psi(t)\mathrm{d}x-\int_{\Omega^T \cup \Omega^H}\psi(0)\mathrm{d}x$ and
$\int_{0}^{t}\int_{\Omega^T \cup \Omega^H}u_{\varepsilon}\cdot\nabla\psi \mathrm{d}%
x\mathrm{d}s=0$. Thus, the weak formulation of the diffuse interface problem
\emph{(\ref{weak form diff prob})} is
equivalent to
\begin{align}
& \int_{0}^{t}\int_{\Omega^T \cup \Omega^H}\{\left(  \phi_{\varepsilon}+1\right)  \partial
_{t}\psi+\left(  \phi_{\varepsilon}+1\right)  u_{\varepsilon} \cdot \nabla%
\psi-\nabla\mu_{\varepsilon}\cdot\nabla\psi\}\mathrm{d}x\mathrm{d}%
s\nonumber\\
& =\int_{\Omega^T \cup \Omega^H}\left(  \phi_{\varepsilon}(t)+1\right)  \psi(t)\mathrm{d}%
x-\int_{\Omega^T \cup \Omega^H}\left(  \phi_{0,\varepsilon}+1\right)  \psi(0)\mathrm{d}%
x.\label{rel1}%
\end{align}
By passing to the limit as $\varepsilon\rightarrow0$ and using the convergence
results of Theorem \emph{\ref{main thm}}, one gets the weak formulation of the
sharp interface problem 
\emph{(\ref{weak form sharp int})}. Moreover, equation
\emph{(\ref{weak form sharp int})} can be formally deduced as follows. Test
equation \emph{(\ref{sharp int delta mu})} with some $\psi\in C_{0}^{\infty}%
(([0,t] \times \Omega))$, multiply \emph{(\ref{sharp nt div u})} by
$(-1+2\chi_{\Omega^{T}})\psi$, and take the sum getting%
\[
0=\int_{0}^{t}\int_{\Omega^{T}\cup\Omega^{H}}\{\Delta\mu\psi-(-1+2\chi
_{\Omega^{T}})  \nabla \cdot u  \psi\}\mathrm{d}x\mathrm{d}s.
\]
By integrating by parts and using equation \emph{(\ref{sharp int V})}
and the boundary conditions on $\mu$ and $u$,
one obtains%
\begin{align}
0 &  =\int_{0}^{t}\int_{\Omega^{T}\cup\Omega^{H}}\{\left(  -\nabla\mu\right)
\nabla\psi+(-1+2\chi_{\Omega^{T}})u\cdot\nabla\psi\}\mathrm{d}x\mathrm{d}%
s\nonumber\\
&  +\int_{0}^{t}\int_{\Sigma}\left(  [\nabla\mu]_{H}^{T}\cdot n-2u\cdot
n\right)  \psi\mathrm{d}S\mathrm{d}s\nonumber\\
&  =\int_{0}^{t}\int_{\Omega^{T}\cup\Omega^{H}}\{\left(  -\nabla\mu\right)
\nabla\psi+(-1+2\chi_{\Omega^{T}}) u \cdot\nabla\psi\}\mathrm{d}%
x\mathrm{d}s+\int_{0}^{t}\int_{\Sigma}\left(  -2V\right)  \psi\mathrm{d}%
S\mathrm{d}s\text{.}\label{ss}%
\end{align}
Interpreting $V$ as the velocity describing the evolution of the interface
$\Sigma$, we intuitively and formally have
\[
\int_{\Sigma}2V\psi\mathrm{d}S=\int_{\Omega^T\cup\Omega^H}\partial_{t}(-1+2\chi_{\Omega^{T}%
})\psi\mathrm{d}x.
\]
By (formally) integrating by parts this relation and substituting into
\emph{(\ref{ss})} and using relations $\int_{\Omega^T\cup\Omega^H} u \cdot\nabla\psi
\mathrm{d}x=0$ and $\int_{0}^{t}\int_{\Omega^T\cup\Omega^H}\partial_{t}\psi\mathrm{d}%
x\mathrm{d}s=\int_{\Omega^T\cup\Omega^H}\psi(t)\mathrm{d}x-\int_{\Omega^T\cup\Omega^H}\psi(0)\mathrm{d}x$,
we get \emph{(\ref{weak form sharp int})}. This suggests that condition
\emph{(\ref{weak form sharp int})} encodes equations \emph{(\ref{sharp int V}%
)}, \emph{(\ref{sharp int delta mu})}, \emph{(\ref{sharp nt div u})},
and the boundary conditions.
\end{enumerate}
\end{remark}

\section{A priori estimates}

\label{priori}

In this section we derive some uniform-in-$\varepsilon$ estimates for
solutions $\left(  u_{\varepsilon},\phi_{\varepsilon},\mu_{\varepsilon
}\right)  $ to system (\ref{diff mod ch})-(\ref{diff mod ic}). In what follows
$C$ will denote a positive constant independent of $\varepsilon$ which
possibly varies even within the same line.

Let $\left(  u_{\varepsilon},\phi_{\varepsilon},\mu_{\varepsilon}\right)  $ be
a solution to system (\ref{diff mod ch})-(\ref{diff mod ic}). Integrating identity (\ref{energy2})
 over $[\tau,t]$ and recalling well preparedness of
initial data (\ref{well preparedness}),
\begin{equation}
E_{\varepsilon}(\phi_{\varepsilon}(t))+\int_{\tau}^{t}\int_{\Omega}|\nabla
\mu_{\varepsilon}|^{2}\mathrm{d}x\mathrm{d}s+\int_{\tau}^{t}\int_{\Omega
}|u_{\varepsilon}|^{2}\mathrm{d}x\mathrm{d}s=E_{\varepsilon}(\phi
_{\varepsilon}(\tau))\leq E_{\varepsilon}(\phi_{0,\varepsilon})\leq
C\text{.}\label{energy identity}%
\end{equation}
Thus, recalling the definition of the energy functional
\[
E_{\varepsilon}(\phi_{\varepsilon})=\int_{\Omega}(\varepsilon \frac{1}{2}|\nabla
\phi_{\varepsilon}|^{2}+\frac{1}{\varepsilon}F(\phi_{\varepsilon}%
))\mathrm{d}x\text{,}%
\]
we have
\begin{equation}
\int_{\Omega}\varepsilon \frac{1}{2}|\nabla\phi_{\varepsilon}(t)|^{2}\mathrm{d}x+\frac
{1}{\varepsilon}F(\phi_{\varepsilon}(t))+\int_{\tau}^{t}\int_{\Omega}|\nabla
\mu_{\varepsilon}|^{2}\mathrm{d}x\mathrm{d}s+\int_{\tau}^{t}\int_{\Omega
}|u_{\varepsilon}|^{2}\mathrm{d}x\mathrm{d}s\leq C\text{.}\label{unif est1}%
\end{equation}
By using $p$-growth of $F$ for large $\phi$ and positivity of $F^{\prime
\prime}(\pm1)$, we get that $F(\phi)\geq\frac{1}{C}(|\phi|-1)^{2}$ for all
$\phi\in%
\mathbb{R}
$. In particular, by using again $p$-growth of $F$, we deduce the estimates:%
\begin{align}
\left\Vert \nabla\mu_{\varepsilon}\right\Vert _{L^{2}(0,\infty;L^{2}(\Omega))}
&  \leq C\text{,}\label{est nabla mu}\\
\left\Vert u_{\varepsilon}\right\Vert _{L^{2}(0,\infty;L^{2}(\Omega))} &  \leq
C\text{,}\label{est u}\\
\left\Vert \varepsilon^{1/2}\nabla\phi_{\varepsilon}\right\Vert _{L^{\infty
}(0,\infty;L^{2}(\Omega))} &  \leq C\text{,}\label{est nabla phi}\\
\int_{\Omega}|\phi_{\varepsilon}(t)|^{p}\mathrm{d}x &  \leq C\text{ for all
}t\geq0\text{,}\label{est phi}\\
\int_{\Omega}\left(  |\phi_{\varepsilon}(t)|-1\right)  ^{2}\mathrm{d}x &
\leq\varepsilon C\text{ for all }t\geq0.\label{est coerc}%
\end{align}
Following \cite{Chen}, we define
\[
W(\phi)=\int_{-1}^{\phi}\sqrt{2\tilde{F}(r)}\mathrm{d}r\text{, where }%
\tilde{F}(r)=\min\{F(r),\max_{z\in\lbrack-1,1]}F(z)+r^{2}\}
\]
and
\[
w_{\varepsilon}(x,t)=W(\phi_{\varepsilon}(x,t))\text{ for a.a. }(x,t)\in
\Omega\times(0,\infty).
\]
Note that by definition $F(r)=\tilde{F}(r)$ for all $r\in\lbrack-1,1]$. By
applying the Young inequality, we easily estimate
\begin{equation}
\int_{\Omega}|\nabla w_{\varepsilon}|\mathrm{d}x=\int_{\Omega}\sqrt{2\tilde
{F}(\phi_{\varepsilon})}|\nabla\phi_{\varepsilon}|\mathrm{d}x\leq
E_{\varepsilon}(\phi_{\varepsilon})\leq C\text{.}\label{est nabla w}%
\end{equation}
In particular, the functions $w_{\varepsilon}$ are uniformly bounded in
$L^{\infty}(0,\infty;W^{1,1}(\Omega))$. We now prove that
\begin{align*}
\left\Vert w_{\varepsilon}\right\Vert _{C^{\frac{1}{16}}([0,\infty
);L^{1}(\Omega))} &  \leq C\text{,}\\
\left\Vert \phi_{\varepsilon}\right\Vert _{C^{\frac{1}{16}}([0,\infty
);L^{2}(\Omega))} &  \leq C\text{.}%
\end{align*}
To this aim let $\rho\in C^{\infty}(%
\mathbb{R}
^{d})$ be any fixed mollifier satisfying
\[
0\leq\rho\leq1\text{ in }B_{1}\text{, \ \ \ }\rho=0\text{ in }%
\mathbb{R}
^{d}\setminus B_{1}\text{, \ \ \ }\int_{B_{1}}\rho\mathrm{d}x=1\text{.}%
\]
For any $\eta_{0}>0$ and any $\eta\in(0,\eta_{0}]$, we define
\[
\phi_{\varepsilon}^{\eta}(x,t)=\int_{B_{1}}\rho(y)\phi_{\varepsilon}(x-\eta
y,t)\mathrm{d}y\text{ for all }x\in\Omega\text{, }t\geq0\text{.}%
\]
Here we have assumed that $\phi_{\varepsilon}$ has been extended to a small
neighborhood of $\Omega$ as follows: for any $x\notin\Omega$ such that
$\mathrm{dist}(x,\Omega)\leq\eta_{0}$, we define
\[
\phi_{\varepsilon}(S+\eta\nu(S),t)=\phi_{\varepsilon}(S-\eta\nu(S),t)\text{
for all } S\in\partial\Omega\text{, }\eta\in\lbrack0,\eta_{0}]\text{, }t\geq0
\]
where $\nu$ denotes the outward normal to $\partial\Omega$. Note that, by
standard properties of mollifiers, we have
\begin{equation}
\left\Vert \nabla\phi_{\varepsilon}^{\eta}(t)\right\Vert _{L^{q}(\Omega)}\leq
C\eta^{-1}\left\Vert \phi_{\varepsilon}(t)\right\Vert _{L^{q}(\Omega)}\leq
C\eta^{-1}\text{ for all }1<q\leq p\text{.}\label{est mollifier}%
\end{equation}
and
\begin{align}
\left\Vert \phi_{\varepsilon}^{\eta}(t)-\phi_{\varepsilon}(t)\right\Vert
_{L^{2}(\Omega)}^{2} &  \leq\int_{\Omega}\int_{B_{1}}|\phi_{\varepsilon
}(x-\eta y,t)-\phi_{\varepsilon}(x,t)|\mathrm{d}x\mathrm{d}y\nonumber\\
&  \leq C\int_{\Omega}\int_{B_{1}}|w_{\varepsilon}(x-\eta y,t)-w_{\varepsilon
}(x,t)|\mathrm{d}x\mathrm{d}y\nonumber\\
&  \leq C\eta\left\Vert \nabla w_{\varepsilon}(t)\right\Vert _{L^{1}(\Omega
)}\leq C\eta\text{.}\label{est}%
\end{align}
Here we have used inequality
\begin{equation}
c_{1}|\phi_{1}-\phi_{2}|\leq|W(\phi_{1})-W(\phi_{2})|\leq c_{2}|\phi_{1}%
-\phi_{2}|(1+|\phi_{1}|+|\phi_{2}|)\text{,}  \label{prop W}
\end{equation}
for all $\phi_{1},\phi_{2}\in%
\mathbb{R}
$ and some positive constant $c_{1},c_{2}$, which follows directly from the
definition of $W$. We fix $0<\tau<t$. Taking the difference of equation
(\ref{weak form diff prob}) at time $t$ and the same equation at time $\tau$,
and using a density argument%
\begin{align*}
\int_{\Omega}\phi_{\varepsilon}(t)\psi\mathrm{d}x-\int_{\Omega}\phi_{\varepsilon
}(\tau)\psi\mathrm{d}x &= \int_{\tau}^{t}\int_{\Omega}(\phi_{\varepsilon}\partial_{t}\psi-\left(\nabla
\mu_{\varepsilon}-u_{\varepsilon}\phi_{\varepsilon}\right)  \nabla\psi )\mathrm{d}x\mathrm{d}s \\
&=\int_{\tau}^{t}\int_{\Omega}(-\left(\nabla
\mu_{\varepsilon}-u_{\varepsilon}\phi_{\varepsilon}\right)  \nabla\psi )\mathrm{d}x\mathrm{d}s
\end{align*}
for all $\psi\in H^{1}_0(\Omega)$. Choosing $\psi=\phi
_{\varepsilon}^{\eta}\left(  t\right)  -\phi_{\varepsilon}^{\eta}\left(
\tau\right)  $, as it is constant in time, we estimate  
\begin{align}
&  \int_{\Omega}\left(  \phi_{\varepsilon}\left(  t\right)  -\phi
_{\varepsilon}\left(  \tau\right)  \right)  \left(  \phi_{\varepsilon}^{\eta
}\left(  t\right)  -\phi_{\varepsilon}^{\eta}\left(  \tau\right)  \right) \mathrm{d}x
\nonumber\\
&  =-\int_{\tau}^{t}\int_{\Omega}\left(  \nabla\mu_{\varepsilon}%
(s)-u_{\varepsilon}(s)\phi_{\varepsilon}(s)\right)  \left(  \nabla
\phi_{\varepsilon}^{\eta}\left(  t\right)  -\nabla\phi_{\varepsilon}^{\eta
}\left(  \tau \right)  \right) \mathrm{d}x  \mathrm{d}s\nonumber\\
&  \leq\left(  \int_{\tau}^{t}\int_{\Omega}|\nabla\phi_{\varepsilon}^{\eta
}\left(  t\right)  -\nabla\phi_{\varepsilon}^{\eta}\left(  \tau\right)
|^{4} \mathrm{d}x \mathrm{d}s\right)  ^{\frac{1}{4}}\left(  \int_{\tau}^{t}\int_{\Omega
}|\nabla\mu_{\varepsilon}(s)-u_{\varepsilon}(s)\phi_{\varepsilon}%
(s)|^{\frac{4}{3}}\mathrm{d}x\mathrm{d}s\right)  ^{\frac{3}{4}}\nonumber\\
&  \leq C(t-\tau)^{\frac{1}{4}}\sup_{s\in(\tau,t)}\left\Vert \nabla
\phi_{\varepsilon}^{\eta}(t)\right\Vert _{L^{4}(\Omega)}\left(  1+\int_{\tau
}^{t}\left\Vert \nabla\mu_{\varepsilon}(s)\right\Vert _{L^{\frac{4}{3}}%
(\Omega)}^{\frac{4}{3}}\mathrm{d}s\right.  \nonumber\\
&  \qquad\qquad\qquad\qquad\qquad\qquad\qquad\quad\left.  +\int_{\tau}%
^{t}\left\Vert \phi_{\varepsilon}(s)\right\Vert _{L^{4}(\Omega)}^{4}\left\Vert
u_{\varepsilon}(s)\right\Vert _{L^{2}(\Omega)}^{2}\mathrm{d}s\right)
^{\frac{3}{4}}\nonumber\\
&  \leq C(t-\tau)^{\frac{1}{4}}\eta^{-1}\left(  1+\left\Vert \nabla
\mu_{\varepsilon}(s)\right\Vert _{L^{2}(0,\infty;L^{2}(\Omega))}^{\frac{4}{3}%
}\right.  \nonumber\\
&  \qquad\qquad\qquad\qquad\qquad\qquad\qquad\quad\left.  +\left\Vert
\phi_{\varepsilon}(s)\right\Vert _{L^{\infty}(0,\infty;L^{4}(\Omega))}%
^{4}\left\Vert u_{\varepsilon}(s)\right\Vert _{L^{2}(0,\infty;L^{2}(\Omega
))}^{2}\right)  ^{\frac{3}{4}}\nonumber\\
&  \leq C(t-\tau)^{\frac{1}{4}}\eta^{-1}\text{.}\label{new est}%
\end{align}
Here we used estimates (\ref{est nabla mu}), (\ref{est u}), (\ref{est phi})
together with $p\geq4$, and (\ref{est mollifier}) for $q=4$. Let now $a,b,c,d$
be real number such that $a=b+c+d$. Then,
\begin{equation}
a^{2}=a(b+c+d)\leq ab+ac+ad\leq ab+\frac{1}{2}a^{2}+c^{2}+d^{2}\text{.}%
\label{dis real num}%
\end{equation}
Using (\ref{dis real num}) for $a=\phi_{\varepsilon}\left(  t\right)
-\phi_{\varepsilon}\left(  \tau\right)  $, $b=\phi_{\varepsilon}^{\eta}\left(
t\right)  -\phi_{\varepsilon}^{\eta}\left(  \tau\right)  $, $c=\phi
_{\varepsilon}\left(  t\right)  -\phi_{\varepsilon}^{\eta}\left(  t\right)  $,
and $d=\phi_{\varepsilon}\left(  \tau\right)  -\phi_{\varepsilon}^{\eta
}\left(  \tau\right)  $, and estimates (\ref{est})-(\ref{new est}), we deduce%
\begin{align*}
\left\Vert \phi_{\varepsilon}(t)-\phi_{\varepsilon}(\tau)\right\Vert
_{L^{2}(\Omega)}^{2} &  \leq2\left\Vert \phi_{\varepsilon}(t)-\phi
_{\varepsilon}^{\eta}(t)\right\Vert _{L^{2}(\Omega)}^{2}+2\left\Vert
\phi_{\varepsilon}(\tau)-\phi_{\varepsilon}^{\eta}(\tau)\right\Vert
_{L^{2}(\Omega)}^{2}\\
&  +2\int_{\Omega}\left(  \phi_{\varepsilon}\left(  t\right)  -\phi
_{\varepsilon}\left(  \tau\right)  \right)  \left(  \phi_{\varepsilon}^{\eta
}\left(  t\right)  -\phi_{\varepsilon}^{\eta}\left(  \tau\right)  \right) \mathrm{d}x \\
&  \leq C\left(  \eta+|t-\tau|^{\frac{1}{4}}\eta^{-1}\right)  \text{.}%
\end{align*}
Choosing $\eta=|t-\tau|^{\frac{1}{8}}$, we get
\[
\left\Vert \phi_{\varepsilon}\right\Vert _{C^{\frac{1}{16}}([0,\infty
);L^{2}(\Omega))}\leq C
\]
and, recalling (\ref{prop W}),
\begin{align*}
\left\Vert w_{\varepsilon}(t)-w_{\varepsilon}(\tau)\right\Vert _{L^{1}%
(\Omega)} &  \leq\left\Vert \phi_{\varepsilon}(t)-\phi_{\varepsilon}%
(\tau)\right\Vert _{L^{2}(\Omega)}^{2}\left(  C+\left\Vert \phi_{\varepsilon
}(t)\right\Vert _{L^{2}(\Omega)}^{2}+\left\Vert \phi_{\varepsilon}%
(\tau)\right\Vert _{L^{2}(\Omega)}^{2}\right)  \\
&  \leq C(t-\tau)^{\frac{1}{16}}\text{,}%
\end{align*}
which implies
\[
\left\Vert w_{\varepsilon}\right\Vert _{C^{\frac{1}{16}}([0,\infty
);L^{1}(\Omega))}\leq C\text{.}%
\]

Starting from the elliptic equation $\mu_{\varepsilon}=-\varepsilon\Delta
\phi_{\varepsilon}+\frac{1}{\varepsilon}F(\phi_{\varepsilon})$, it is possible
to derive uniform estimates for $\mu_{\varepsilon}$ and for the discrepancy
density%
\[
\xi_{\varepsilon}(\phi_{\varepsilon})=\frac{\varepsilon}{2}|\nabla
\phi_{\varepsilon}|^{2}-\frac{1}{\varepsilon}F(\phi_{\varepsilon})\text{.}%
\]

\begin{lemma}
\textbf{\emph{\cite[Lemma 3.4]{Chen}}} \label{lemma estimates for mu}There
exist positive constants $C$ and $\varepsilon_{0}$ such that for every $t$ and
$\varepsilon\in(0,\varepsilon_{0})$ the following holds%
\[
\left\Vert \mu_{\varepsilon}(t)\right\Vert _{H^{1}(\Omega)}\leq C\left(
E_{\varepsilon}(t)+\left\Vert \nabla\mu_{\varepsilon}(t)\right\Vert
_{L^{2}(\Omega)}\right)  \text{.}%
\]
In particular, for every $s>0$ there exists a positive constant $C(s)$, such
that $\left\Vert \mu_{\varepsilon}(t)\right\Vert _{L^{2}(0,s;L^{2}(\Omega
))}\leq C(s)$.
\end{lemma}

\begin{lemma}
\textbf{\emph{\cite[Theorem 3.6]{Chen}}}
\label{lemma estimate for discrepancy}There exist a positive constant
$\eta_{0}\in(0,1]$ and continuous nondecreasing functions $M_{1}(\eta)$ and
$M_{1}(\eta)$ defined on $[0,\eta_{0})$ such that for all $\varepsilon
\in\left(  0,\frac{1}{M_{1}(\eta)}\right)  $ and all $t>0$, we have
\[
\int_{0}^{t}\int_{\Omega}\left(  \xi_{\varepsilon}(\phi_{\varepsilon})\right)
^{+}\mathrm{d}x\mathrm{d}s\leq\eta\int_{0}^{t}E_{\varepsilon}(\phi
_{\varepsilon}  )\mathrm{d}x\mathrm{d}s+\varepsilon M_{2}%
(\eta)\int_{0}^{t}\int_{\Omega}|\mu_{\varepsilon}|^{2}\mathrm{d}%
x\mathrm{d}s\text{.}%
\]
In particular,%
\begin{equation}
\lim_{\varepsilon\rightarrow0}\int_{0}^{t}\int_{\Omega}\left(  \xi
_{\varepsilon}(\phi_{\varepsilon})\right)  ^{+}\mathrm{d}x\mathrm{d}%
s=0\text{.} \label{discrepancy goes negative}%
\end{equation}

\end{lemma}

\section{Convergence}

\label{convergence}

Starting from the above uniform estimates, we now deduce some convergence results.

\begin{lemma}
For every sequence $\varepsilon\rightarrow0$, there exists a (not relabeled)
subsequence and a nonincreasing function $E$, such that
\[
E_{\varepsilon}(\phi_{\varepsilon}(t))\rightarrow E(t)\text{ for all }%
t\geq0\text{.}%
\]

\end{lemma}

\begin{proof}
Define $E_{\varepsilon}(t)=E_{\varepsilon}(\phi_{\varepsilon}(t))$. Note that
$E_{\varepsilon}(t)$ is uniformly bounded as a consequence of identity
(\ref{energy identity}). Furthermore, the sequence $E_{\varepsilon}(\cdot)$ is
uniformly continuous as a consequence of monotonicity, of the energy identity
(\ref{energy identity}), and of the uniform bounds of $\nabla\mu_{\varepsilon
}$ and $u_{\varepsilon}$ in $L^{2}(0,T;L^{2}(\Omega))$. Thus, the statement of
the lemma follows by applying the Ascoli-Arzel\`{a} theorem.
\end{proof}

\begin{lemma}
\label{lemma convergence}For every sequence $\varepsilon\rightarrow0$, there
exists a (not relabeled) subsequence and a set $\Omega^{T}\subset\Omega
\times\lbrack0,\infty)$, such that
\begin{align*}
w_{\varepsilon} &  \rightarrow2\sigma\chi_{\Omega^{T}}\text{ a.e. in }%
\Omega\times\lbrack0,\infty)\text{ }\text{and in }C^{\frac{1}{17}}%
([0,t];L^{1}\left(  \Omega\right)  )\text{ for all }t>0\text{,}\\
\phi_{\varepsilon} &  \rightarrow-1+2\chi_{\Omega^{T}}\text{ a.e. in }%
\Omega\times\lbrack0,\infty)\text{ }\text{and in }C^{\frac{1}{17}}%
([0,t];L^{2}\left(  \Omega\right)  )\text{ for all }t>0\text{,}\\
\mu_{\varepsilon} &  \rightarrow\mu\text{ weakly in }L_{\mathrm{loc}}%
^{2}(0,\infty;H^{1}(\Omega))\text{,}\\
u_{\varepsilon} &  \rightarrow u\text{ weakly in }L^{2}(0,\infty
;L_{\operatorname{div}}^{2}(\Omega))\text{.}%
\end{align*}
Moreover,%
\begin{align*}
\int_{\Omega}|\chi_{\Omega_{t}^{T}}-\chi_{\Omega_{\tau}^{T}}|\mathrm{d}x &
\leq C|t-\tau|^{\frac{1}{8}}\text{ for any }0\leq\tau<t\text{,}\\
|\Omega_{t}^{T}| &  =|\Omega_{0}^{T}|\text{ for any }t\geq0\text{,}\quad
\chi_{\Omega^{T}}\in L^{\infty}(0,\infty;\mathrm{BV}(\Omega))\text{ and }\\
2\sigma|\mathrm{D}\chi_{\Omega_{t}^{T}}|(\Omega) &  \leq E(t)\leq E(0)\text{.}%
\end{align*}

\end{lemma}

\begin{proof}
As $\left\Vert w_{\varepsilon}\right\Vert _{L^{\infty}(0,\infty;W^{1,1}%
(\Omega))}+\left\Vert w_{\varepsilon}\right\Vert _{C^{\frac{1}{16}}%
([0,\infty);L^{1}(\Omega))}\leq C$ and $W^{1,1}\left(  \Omega\right)  $ is
compactly embedded in $L^{1}\left(  \Omega\right)  $, there exists a (not
relabeled) sequence $\varepsilon\rightarrow0$ such that
\[
w_{\varepsilon}\rightarrow w\text{ a.e. in }\Omega\times\lbrack0,\infty)\text{
and in }C^{\frac{1}{17}}([0,t];L^{1}\left(  \Omega\right)  )\text{ for all
}t>0\text{,}%
\]
for some limit $w\in C^{\frac{1}{17}}([0,t];L^{1}\left(  \Omega\right)  )$
(cf. \cite[Prop. 1.1.4]{Lu} and \cite[Thm 4.4]{Amm}). Recalling the definition
of $w_{\varepsilon}$ and estimate (\ref{prop W}), we conclude that there
exists $\phi\in C^{\frac{1}{17}}([0,t];L^{2}\left(  \Omega\right)  )$ such
that
\[
\phi_{\varepsilon}\rightarrow\phi\text{ a.e. in }\Omega\times\lbrack
0,\infty)\text{ and in }C^{\frac{1}{17}}([0,t];L^{2}\left(  \Omega\right)
)\text{ for all }t>0\text{.}%
\]
As a consequence of estimate (\ref{est coerc}), we deduce%
\[
\int_{\Omega}\left(  |\phi_{\varepsilon}|-1\right)  ^{2}\mathrm{d}x\leq
C\int_{\Omega}F(\phi_{\varepsilon})\mathrm{d}x\leq\varepsilon C\text{. }%
\]
Thus, the limit $\phi$ takes values in $\{-1,1\}$. In particular, there exists
a set $\Omega^{T}\subset\Omega\times\lbrack0,\infty)$ such that
\[
\phi=-1+2\chi_{\Omega^{T}}\text{.}%
\]
Hence, by definition of $w_{\varepsilon}$ and continuity of $\tilde{F}$, we
get
\[
w=\int_{-1}^{\phi}\sqrt{2\tilde{F}(r)}\mathrm{d}r=2\sigma\chi_{\Omega^{T}%
}\text{,}%
\]
where $\sigma=\int_{-1}^{1}\sqrt{\frac{1}{2}\tilde{F}(r)}\mathrm{d}r=\int
_{-1}^{1}\sqrt{\frac{1}{2}F(r)}\mathrm{d}r$. Here we used the fact that
$F(r)=\tilde{F}(r)$ for $r\in\lbrack-1,1]$, which directly follows from the
definition of $\tilde{F}$. Let now $\Omega_{t}^{T}=\{x\in\Omega:(x,t)\in
\Omega^{T}\}$. Then, for every $0\leq\tau<t$, we have
\begin{align*}
\int_{\Omega}|\chi_{\Omega_{t}^{T}}-\chi_{\Omega_{\tau}^{T}}|\mathrm{d}x &
=\int_{\Omega}|\chi_{\Omega_{t}^{T}}-\chi_{\Omega_{\tau}^{T}}|^{2}%
\mathrm{d}x=\lim_{\varepsilon\rightarrow0}\frac{1}{4}\int_{\Omega}%
|\phi_{\varepsilon}(t)-\phi_{\varepsilon}(\tau)|^{2}\mathrm{d}x\\
&  \leq C|t-\tau|^{\frac{1}{8}}\text{.}%
\end{align*}
As a consequence of the mass conservation
\[
\int_{\Omega}\phi_{\varepsilon}(t)\mathrm{d}x=\int_{\Omega}\phi_{0}%
\mathrm{d}x=m_{0}|\Omega|\text{,}%
\]
we have
\[
|\Omega_{t}^{T}|=\int_{\Omega}\chi_{\Omega_{t}^{T}}\mathrm{d}x=\lim
_{\varepsilon\rightarrow0}\frac{1}{2}\int_{\Omega}\left(  \phi_{\varepsilon
}(t)+1\right)  \mathrm{d}x=\frac{m_{0}+1}{2}|\Omega|=|\Omega_{0}^{T}|\text{.}%
\]
Moreover, as a consequence of estimate (\ref{est nabla w}), we have
$|Dw_{\varepsilon}(t)|(\Omega)=\left\Vert \nabla w_{\varepsilon}(t)\right\Vert
_{L^{1}\left(  \Omega\right)  }\leq E_{\varepsilon}(\phi_{\varepsilon}(t))$.
Taking the liminf for $\varepsilon\rightarrow0$ and using the lower
semicontinuity of the $\mathrm{BV}$ norm, we conclude
\[
2\sigma|\mathrm{D}\chi_{\Omega_{t}^{T}}|(\Omega)\leq|\mathrm{D}w|(\Omega)\leq
E(t)\text{.}%
\]
Finally, convergences
\begin{align*}
\mu_{\varepsilon} &  \rightarrow\mu\text{ weakly in }L_{\mathrm{loc}}%
^{2}(0,\infty,;H^{1}(\Omega))\text{,}\\
u_{\varepsilon} &  \rightarrow u\text{ weakly in }L^{2}(0,\infty;L^{2}%
(\Omega))
\end{align*}
follows directly from Lemma \ref{lemma estimates for mu} and estimate
(\ref{energy identity}) respectively.
\end{proof}

As a consequence of estimate (\ref{energy identity}) and (\ref{est nabla phi}%
), we have that convergences (\ref{conv energy density}) and
(\ref{conv grad phi}) hold for some limit measures $\lambda$ and $\lambda
_{ij}$. Thus, we proved the convergence results stated in Theorem
\ref{main thm}. We now construct the varifold
$V$ and show that the limits $\mu$, $u$, $\lambda$, and
$\lambda_{ij}$ solve the sharp-interface problem.

We first note that, for any $0\leq\tau<
t$, we have
\[
\int_{\tau}^{t}\int_{\Omega}\mathrm{d}\lambda(x,s)=\lim_{\varepsilon
\rightarrow0}\int_{\tau}^{t}\int_{\Omega}e_{\varepsilon}(\phi_{\varepsilon
})\mathrm{d}x\mathrm{d}s=\int_{\tau}^{t}E(s)\mathrm{d}s\text{.}%
\]
Moreover, $\lambda$ can be decomposed (in the sense of Radon measures) as follows
\[
\mathrm{d}\lambda(x,t)=\mathrm{d}\lambda^{t}(x)\mathrm{d}t\text{,}%
\]
where $\lambda^{t}(\bar{\Omega})=E(t)$ for a.a. $t\in(0,\infty)$. In
particular, using relation (\ref{energy identity}) and the weak lower
semicontinuity of the norm, we obtain
\begin{align*}
\lambda^{t}(\bar{\Omega}) &  =E(t)=\lim_{\varepsilon\rightarrow0}%
E_{\varepsilon}(t)\\
&  \leq-\liminf_{\varepsilon\rightarrow0}\left\{  \int_{\tau}^{t}\int_{\Omega
}|\nabla\mu_{\varepsilon}|^{2}\mathrm{d}x\mathrm{d}s+\int_{\tau}^{t}%
\int_{\Omega}|u_{\varepsilon}|^{2}\mathrm{d}x\mathrm{d}s\right\}
+\lim_{\varepsilon\rightarrow0}E_{\varepsilon}(\phi_{\varepsilon}(\tau))\\
&  \leq-\int_{\tau}^{t}\int_{\Omega}|\nabla\mu|^{2}\mathrm{d}x\mathrm{d}%
s-\int_{\tau}^{t}\int_{\Omega}|u|^{2}\mathrm{d}x\mathrm{d}s+E(\tau
)=\lambda^{\tau}(\bar{\Omega})\text{,}%
\end{align*}
which is equivalent to (\ref{nergy ineq sharp int}). Moreover, as a
consequence of condition $2\sigma|\mathrm{D}\chi_{\Omega_{t}^{T}}|(\Omega)\leq
E(t)$ obtained in Lemma \ref{lemma convergence}, we deduce estimate
(\ref{bound}). Next we study the relation between $\lambda_{ij}$ and $\lambda
$.\ Let $Y,Z\in C\left(  \bar{\Omega}\times\lbrack0,t];%
\mathbb{R}
^{d}\right)  $ and observe that%
\[
\int_{0}^{t}\int_{\Omega}Y\cdot(\varepsilon\nabla\phi_{\varepsilon}%
\otimes\nabla\phi_{\varepsilon})\cdot Z\mathrm{d}x\mathrm{d}s \leq\int_{0}%
^{t}\int_{\Omega}|Y||Z|e_{\varepsilon}(\phi_{\varepsilon})\mathrm{d}%
x\mathrm{d}s+\int_{0}^{t}\int_{\Omega}|Y||Z|\xi_{\varepsilon}(\phi
_{\varepsilon})\mathrm{d}x\mathrm{d}s\text{.}%
\]
Using Lemma \ref{lemma estimate for discrepancy}, we have that
\[
\lim_{\varepsilon\rightarrow0}\int_{0}^{t}\int_{\Omega}|Y||Z|\xi_{\varepsilon
}(\phi_{\varepsilon})\mathrm{d}x\mathrm{d}s\leq0.
\]
Hence, taking the limit for $\varepsilon\rightarrow0$, we get
\begin{equation}
\int_{0}^{t}\int_{\Omega}Y\cdot(\mathrm{d}\lambda_{ij}(x,s))_{d\times d}\cdot
Z\leq\int_{0}^{t}\int_{\Omega}|Y||Z|\mathrm{d}\lambda(x,s)\text{.}%
\label{formula}%
\end{equation}
Thus, $\lambda_{ij}$ are absolutely continuous with respect of $\lambda$ in
the sense of measures. Consequently, we can define the Radon-Nikodin
derivative of $\lambda_{ij}$ with respect to $\lambda$ as a $\lambda
$-measurable function $v_{ij}$ such that
\[
\mathrm{d}\lambda_{ij}(x,t)=v_{ij}(x,t)\mathrm{d}\lambda(x,t)\text{
\ \ \ \ \ }\lambda\text{-a.e.}%
\]
From formula (\ref{formula}), it follows that
\[
0\leq\left(  v_{ij}\right)  _{d\times d}\leq I\text{ \ \ \ \ }\lambda
\text{-a.e. }%
\]
and that
\[
\left(  v_{ij}\right)  _{d\times d}=\sum_{i=1}^{d}c_{i}v_{i}\otimes
v_{i}\text{ \ \ \ }\lambda\text{-a.e.}%
\]
for some $\lambda$-measurable functions $c_{i}$ and unit vectors $v_{i}$,
$i=1,...,d$. Moreover, they satisfy%
\[
0\leq c_{i}\leq1\text{,}\quad\sum_{i=1}^{d}c_{i}\leq1,\quad\sum_{i=1}^{d}%
v_{i}\otimes v_{i}=I.
\]
In order to construct the varifold $V$, we observe that, by multiplying equation
\[
\mu_{\varepsilon}=-\varepsilon\Delta\phi_{\varepsilon}+\frac{1}{\varepsilon
}F^{\prime}(\phi_{\varepsilon})
\]
with $Y\cdot\nabla\phi_{\varepsilon}$ for some $Y\in C^{1}(\bar{\Omega};%
\mathbb{R}
^{d})$ and integrating over $\Omega$, we get
\begin{align}
\int_{\Omega}Y\cdot\nabla\phi_{\varepsilon}\mu_{\varepsilon}\mathrm{d}x &
=\int_{\Omega}Y\cdot\nabla\phi_{\varepsilon}\left(  -\varepsilon\Delta
\phi_{\varepsilon}+\frac{1}{\varepsilon}F^{\prime}(\phi_{\varepsilon})\right)
\mathrm{d}x\nonumber\\
&  =-\int_{\Omega}\nabla Y:\left(  e_{\varepsilon}(\phi_{\varepsilon
})I-\varepsilon\nabla\phi_{\varepsilon}\otimes\nabla\phi_{\varepsilon}\right)
\mathrm{d}x+\int_{\partial\Omega}e_{\varepsilon}\left(  \phi_{\epsilon
}\right)  Y\cdot\nu\mathrm{d}S\text{.}\label{eqq}%
\end{align}
By passing to the limit for $\varepsilon\rightarrow0$ in relation (\ref{eqq}), we
obtain, for every $t>0$ and $Y\in C_{0}^{1}\left(  \Omega;%
\mathbb{R}
^{d}\right)  $,%
\begin{align*}
2\int_{\Omega}\chi_{\Omega_{t}^{T}}\nabla\cdot(\mu(t)Y)\mathrm{d}x &
=\int_{\Omega}\nabla Y:\left(  I-\sum_{i=1}^{d}c_{i}(x,t)v_{i}(x,t)\otimes
v_{i}(x,t)\right)  \mathrm{d}\lambda^{t}(x)\\
&  =\int_{\Omega}\nabla Y:\sum_{i=1}^{d}b_{i}^{t}(x)\left(  I-v_{i}%
(x,t)\otimes v_{i}(x,t)\right)  \mathrm{d}\lambda^{t}(x)
\end{align*}
where the coefficients $b_{i}^{t}$ are given by%
\[
b_{i}^{t}(x)=c_{i}(x,t)+\frac{1}{d-1}\left(  1-\sum_{i=1}^{d}c_{i}%
(x,t)\right)  \text{. }%
\]
Note that%
\[
0\leq b_{i}^{t}\leq1\text{,}\quad\sum_{i=1}^{d}b_{i}^{t}\geq1\text{.}%
\]
Finally, we define $p_{i}^{t}\in P$ by
\[
p_{i}^{t}=\{v_{i}(x,t),-v_{i}(x,t)\}\text{,}%
\]
$V^{t}$ as in (\ref{def V}), and $V$ by
\[
\mathrm{d}V(x,t,p)=\mathrm{d}V^{t}(x,p)\mathrm{d}t\text{.}%
\]
Moreover, by construction $V$ satisfies conditions (\ref{first variation V})
and (\ref{first variation V2}).

Relation (\ref{weak form sharp int}) follows from (\ref{rel1}) by passing to
the limit \ $\varepsilon\rightarrow0$, and using the above convergences.

We are only left to show relation (\ref{weak form u}). To this aim, let
$\varphi\in C_{0,\operatorname{div}}^{\infty}\left(  \Omega;%
\mathbb{R}
^{d}\right)  $. Then, by using relation $u_{\varepsilon}=-\nabla
P+\mu_{\varepsilon}\nabla\phi_{\varepsilon}$, for every $t>0$, we have
\[
\int_{0}^{t}\int_{\Omega}u_{\varepsilon}\varphi\mathrm{d}x\mathrm{d}%
s=-\int_{0}^{t}\int_{\Omega}\nabla\mu_{\varepsilon}\cdot\varphi\phi
_{\varepsilon}\mathrm{d}x\mathrm{d}s\text{.}%
\]
The above convergence results allow us to pass to the limit for $\varepsilon
\rightarrow0$ getting%
\[
\int_{0}^{t}\int_{\Omega}u\varphi\mathrm{d}x\mathrm{d}s=-\int_{0}^{t}%
\int_{\Omega}\nabla\mu\cdot\varphi\left(  -1+2\chi_{\Omega_{s}^{T}}\right)
\mathrm{d}x\mathrm{d}s\text{.}%
\]
Integrating by parts the right-hand side, we obtain%
\[
\int_{0}^{t}\int_{\Omega}u\varphi\mathrm{d}x\mathrm{d}s=\int_{0}^{t}%
\int_{\Sigma_{s}}2\mu\varphi\mathrm{d}S\mathrm{d}s\text{.}%
\]
This concludes the proof of Theorem \ref{main thm}.

\section*{Acknowledgements}

The financial support of the FP7-IDEAS-ERC-StG \#256872 (EntroPhase) is
gratefully acknowledged by the authors. The present paper also benefits from
the support of the GNAMPA (Gruppo Nazionale per l'Analisi Matematica, la
Probabilit\`{a} e le loro Applicazioni) of INdAM (Istituto Nazionale di Alta
Matematica) and the IMATI -- C.N.R. Pavia. S.M. acknowledges support by the
Austrian Science Fund (FWF) project P27052-N25. The Authors would like to
acknowledge the kind hospitality of the Erwin Schr\"{o}dinger International
Institute for Mathematics and Physics, where part of this research was
developed under the frame of the Thematic Program Nonlinear Flows.

\end{document}